\documentclass[a4paper,UKenglish]{amsart}
\usepackage{url,amscd,epsfig,enumerate,graphicx,amsfonts,latexsym}
\usepackage{amssymb,amscd}
\usepackage{amsfonts,mathrsfs}
\renewcommand{\leq}{\leqslant}
\renewcommand{\geq}{\geqslant}
\usepackage{subcaption}
\usepackage{hyperref}

\newcommand{\sep}{\wr}

\newtheorem*{notation}{Notation}

\newenvironment{proof*}[1]
  {\renewcommand{\proofname}{#1}%
   \begin{proof}}
  {\end{proof}}

\newcommand{\PSPACE}{\ensuremath{\mathsf{PSPACE}}}
\newcommand{\NSPACE}{\ensuremath{\mathsf{NSPACE}}}

\usepackage{xcolor}

\newcommand{\DG}{Dahmani and Guirardel}

\newcommand{\N}{\mathbb N}

\title[Solutions to equations in  hyperbolic groups]{Solutions sets to systems of equations in hyperbolic groups 
are EDT0L in PSPACE}

\author{Laura Ciobanu}
\address{School of Mathematical and Computer Sciences,
 Heriot-Watt University, 
 Edinburgh EH14 4AS,
 Scotland}
\email{l.ciobanu@hw.ac.uk}

\author{Murray Elder}
\address{University of Technology Sydney, Ultimo NSW 2007, Australia}
\email{murray.elder@uts.edu.au}

\date{\today}

\newtheorem{theorem}{Theorem}
\newtheorem{proposition}[theorem]{Proposition}
\newtheorem{lemma}[theorem]{Lemma}
\newtheorem{corollary}[theorem]{Corollary}

\theoremstyle{defn}
\newtheorem{definition}[theorem]{Definition}
\newtheorem{remark}[theorem]{Remark}


\thanks{Research supported by Australian Research Council (ARC) Project DP160100486,   EPSRC grant EP/R035814/1 and a Follow-On Grant from the International Centre of Mathematical Sciences (ICMS), Edinburgh}

\begin{document}

\maketitle

\begin{abstract}
We show that the full set of solutions to systems of equations and inequations in a hyperbolic group, with or without torsion,
 as shortlex geodesic words, 
  is  an EDT0L language whose 
  specification can be  computed in \NSPACE$(n^2\log n)$ for the torsion-free case and \NSPACE$(n^4\log n)$ in the torsion case. 
  Our work combines deep geometric results by Rips, Sela, \DG\ on decidability of existential theories of  hyperbolic groups, work of computer scientists including Plandowski, Je\.z, Diekert and others  on \PSPACE\ algorithms to solve equations in free monoids and groups using compression, and an intricate language-theoretic analysis.

The present work gives an essentially optimal formal language description for all solutions in all hyperbolic groups, and an explicit and surprising low space complexity to compute them.
\end{abstract}

\section{Introduction}

Hyperbolic groups were introduced by Gromov in 1987 \cite{Gromov}, and play a significant role in group theory and geometry \cite{isomorphismDG,definable,SelaElem}. Virtually free groups, small cancellation groups, and the fundamental groups of extensive classes of negative curvature manifolds are important examples (see  \cite{MSRInotes} for background). 
In a certain probabilistic sense made precise in \cite{GromovRandom,OlS,Sil}, almost all finitely generated groups are 
hyperbolic. They admit very  efficient solutions to the word and conjugacy problems \cite{EpHoltConj,realtime, HoltLS}, and extremely nice language-theoretic properties, for example the set of all geodesics 
 over any generating set is regular (see Lemma~\ref{prop:reg-hyp}), and forms a biautomatic structure \cite{WordProc}. They are exactly the groups which admit context-free multiplication tables \cite{GilmanHyp}, and have a particularly simple characterisation in terms of rewriting systems \cite{Cannon, Lys} (see Lemma~\ref{lem:Dehn}).

In this paper we consider 
systems of
 equations and inequations in hyperbolic groups,
 building on and generalising work recently done 
 in the area of solving equations over various groups and monoids in  \PSPACE. Starting with work of Plandowski  \cite{Plandowski}, many prominent researchers have given \PSPACE\ algorithms \cite{CDE, DEicalp,dgh01, DiekJezK,MR3571087, Jez2,Jez1} to find (all) solutions to systems of equations  
 over free monoids, free groups, partially commutative monoids and groups, and virtually free groups (that is, groups which have a free subgroup of finite index). 

The satisfiability of equations over torsion-free hyperbolic groups is decidable by the work of Rips and Sela \cite{RS95}, who reduced the problem in hyperbolic groups to solving equations in free groups, and then calling on Makanin's algorithm \cite{mak83a}. 
Kufleitner proved  \PSPACE\ for decidability  in the torsion-free case \cite{DIP-1922}, without an explicit complexity bound, by following Rips-Sela and then using Plandowski's result  \cite{Plandowski}.  
\DG\ radically extended Rips and Sela's work to all
 hyperbolic groups (with torsion), by reducing systems of equations to systems over virtually free groups, which they then reduced to systems of {\em twisted} equations over free monoids \cite{DG}.
In terms of describing solution sets, 
 Grigorchuck and Lysionok gave efficient algorithms for the special case of quadratic equations   \cite{GLquadratic}.

Here we combine  Rips, Sela, \DG's approach with recent work of the authors with Diekert \cite{CDE,DEicalp,DEarxiv} to obtain the following results. \begin{theorem}[Torsion-free]\label{thmTorsionFree}
Let $G$ be a torsion-free hyperbolic group with finite symmetric generating set $S$. 
Let $\Phi$ be a system of equations and inequations
of size $n$ (see Section~\ref{sec:notationSolns} for a precise definition of  input size). 
Then the set of all solutions, as tuples of shortlex geodesic words over $S$, 
is EDT0L. Moreover there is 
 an $\NSPACE(n^2\log n)$ algorithm which on input $\Phi$ prints a description for the EDT0L grammar.
\end{theorem}
\begin{theorem}[Torsion]\label{thmTorsion}
Let $G$ be a  hyperbolic group with torsion, with finite symmetric generating set $S$. 
Let $\Phi$ be a system of equations and inequations
of size $n$  (see Section~\ref{sec:notationSolns} for a precise definition of  input size). 
Then the set of all solutions, as tuples of shortlex geodesic words over $S$, 
 is EDT0L. Moreover there is 
 an $\NSPACE(n^4\log n)$ algorithm which on input $\Phi$  prints a description for the EDT0L grammar.
\end{theorem}
A corollary of Theorems 1 and 2 is that the existential theory for hyperbolic groups can be decided in $\NSPACE(n^2\log n)$ for torsion-free and $\NSPACE(n^4\log n)$ for groups with torsion.  
Another consequence of our work is that we can decide in the same space complexity as above whether or not the solution set is empty, finite or infinite.

EDT0L is a surprisingly low language complexity for this problem. EDT0L languages are playing an increasingly  useful role in group theory, not only in describing solution sets to equations in  groups \cite{CDE,DEicalp,DiekJezK}, but more generally \cite{BEboundedLATA,BCEZ,CEF}.

The paper is organised  as follows. We briefly set up some notation for solution sets and input size in  Section~\ref{sec:notationSolns}. 
We then give an informal description of the entire argument for the torsion-free case in Section~\ref{sec:overview}. This overview uses various concepts which are defined more carefully afterwards, but we hope that having the entire argument in one place is useful for the reader to understand the `big picture' before descending into the details.
Section~\ref{sec:EDT0L} develops necessary material on EDT0L and space complexity. 
Section~\ref{sec:hyp-intro} covers the necessary background on hyperbolic groups, including the key step to obtain a full solution set (as tuples of  shortlex geodesics)  from a {\em covering solution set} (see Definition \ref{solutiondef}(iii)).
In 
Section~\ref{sec:torsionfree} we use Rips and Sela's \emph{canonical representatives} (see  Appendix \ref{sec:torsionfree})
 in torsion-free hyperbolic groups, to reduce the problem of finding solutions in a torsion-free hyperbolic group to finding solutions in the free group on the same generators as the hyperbolic one. We show that if the input system has size $n$ then the resulting system in the free group has size $O(n^2)$. Applying \cite{CDE} produces a  covering solution set in $O(n^2\log n)$ nondeterministic space, from which we obtain the 
 full set of solutions as  shortlex geodesics in the original group, as an EDT0L language, in the same space complexity.
 In Section~\ref{sec:torsion} we prove the general case for hyperbolic groups with torsion, following  \DG\ who construct canonical representatives in a graph containing the Cayley graph of the hyperbolic group, and working in an associated virtually-free group. 

 \section{Notations for equations and solution sets}\label{sec:notationSolns}
Let $G$ be a fixed  group with finite symmetric generating set $S$. Let $\pi\colon S^*\to G$ be the natural projection map.
Let $\{X_1, \dots, X_{m}\}$, $m\geq 1$, be a set of variables to which we adjoin their formal inverses $X_i^{-1}$ and denote by $\mathcal X$ the union $\{X_i, X_i^{-1} \mid 1 \leq i \leq m\}$. Let $\mathcal{C}=\{a_1, \dots, a_k\} \subseteq G$ be a set of constants and
\begin{equation}\label{system}
\Phi=\{\varphi_j(\mathcal X, \mathcal{C})=1\}_{j=1}^h \cup \{\varphi_j(\mathcal X, \mathcal{C})\neq1\}_{j=h+1}^s
\end{equation}
be a set of $s$ equations and inequations in $G$, where the length of each (in)equation is $l_i$. Then the total length of the equations is $n=\sum_{i=1}^s l_i$, and we take $|\Phi|=n$ as the input size in the remainder of the paper.

A tuple $(g_1,\dots, g_m)\in G^m$ \textit{solves} an equation  [resp. inequation] $\varphi_j$ in $\Phi$ if replacing each variable $X_i$ by $g_i$ (and $X_{i}^{-1}$ by $g_i^{-1}$) produces an identity [resp. inequality] in the group as follows:
$$\varphi_j(g_1, \dots, g_m, a_1, \dots, a_k)=1 \ [\text{resp. }\varphi_j(g_1, \dots, g_m, a_1, \dots, a_k)\neq1] .$$ A tuple $(g_1,\dots, g_m)\in G^m$ solves $\Phi$ if it simultaneously
solves $\varphi_j$ for all $1\leq j\leq s$.

\begin{definition}\label{solutiondef}
\begin{itemize}
\item[(i)] The {\em group element solution set} to $\Phi$ is the set $$\text{Sol}_G(\Phi)=\{(g_1,\dots, g_m) \in G^m \mid (g_1,\dots, g_m) \text{ solves } \Phi\}.$$
\item[(ii)] Let $T\subseteq S^*$ and $\#$ a symbol not in $S$.  The  {\em full set of $T$-solutions} is the set 
 $$\text{Sol}_{T,G}(\Phi)=\{w_1\#\dots\# w_m  \mid w_i \in T,  (\pi(w_1),\dots, \pi(w_m) ) \text{ solves  } \Phi\}.$$ 
\item[(iii)]
A set
$L\subseteq\{w_1\#\dots\# w_m \mid w_i \in S^*, 1\leq i \leq m\}$ is a {\em covering solution set} to $\Phi$ if  \[\{(\pi(w_1),\dots, \pi(w_k))\mid w_1\#\dots\# w_m \in L \}=\text{Sol}_G(\Phi).\] 
\end{itemize}
\end{definition}

\section{Overview of the proof}\label{sec:overview}

In a free group, the equation $xy=z$ 
 has a solution in reduced words (that is, words which do not contain factors $aa^{-1}$ for any $a\in S$) 
 if and only if there exist words $P,Q,R$  with $x=PQ, y=Q^{-1}R, z=PR$  in the free monoid with involution over $S$  (\cite[Lemma 4.1]{CDE}). In a hyperbolic group this direct reduction to cancellation-free equations is no longer true: a triangle $xy=z$ where $x,y,z$ are replaced by geodesics  looks as in Figure~\ref{subfig:geod}.

\begin{figure}[ht]
    \centering
    \begin{subfigure}[t]{0.45\textwidth}
        \centering
        \includegraphics[height=1.2in]{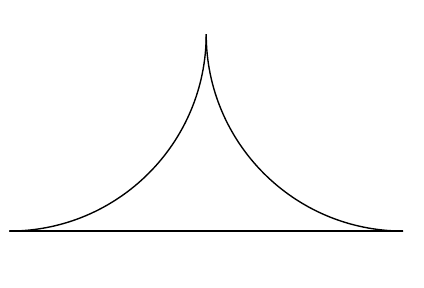}

        \caption{Using geodesics}\label{subfig:geod}
    \end{subfigure}%
    ~ 
    \begin{subfigure}[t]{0.45\textwidth}
        \centering
        \includegraphics[height=1.2in]{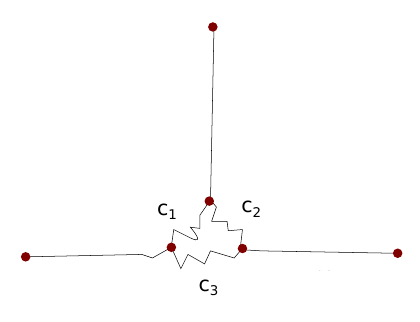}
        \caption{Using canonical representatives}\label{subfig:can}
    \end{subfigure}
    \caption{Solutions to $xy=z$ in the Cayley graph of a hyperbolic group.}\label{fig:triangle}
\end{figure}

Rips and Sela \cite{RS95} proved 
 that in a torsion-free hyperbolic group one can define certain special words called {\em canonical representatives} 
 so that a system of equations of the form $X_jY_j=Z_j, 1\leq j\leq O(n)$   has solutions which are canonical representatives with the properties that their prefixes and suffixes coincide, as shown in Figure~\ref{subfig:can}, and the inner circle is the concatenation of three  words with lengths in 
 $O(n)$.
Moreover, these canonical representatives are $(\lambda,\mu)$-quasigeodesics (Definition~\ref{defn:qg}) where the constants $\lambda, \mu$ depend only on the group. 

We use these facts to devise the following algorithm, presented here for the torsion-free case. We treat the hyperbolic group $G$ with finite generating set $S$ as a constant.
On input a system of equations and inequations as in (\ref{system}) of size $n$: 
\begin{enumerate}\item Replace inequations by equations
(by using a new variable and requiring that this variable is not trivial in the group, as explained in Section 6.3). 
\item Triangulate the system, so that all equations have the form $X_jY_j=Z_j$. The size of the resulting system is still in $O(n)$. Suppose there are $q\in O(n)$ such equations.
\item Enumerate, one at a time, all  
possible tuples $\mathbf c=(c_{11},c_{12},c_{13},\dots, c_{q1},c_{q2},c_{q3})$ of words (say, in lex order)  so that the length $\ell(c_{ji})$ with respect to $S$ is bounded by a constant  in $O(n)$. 
Note that the size of each tuple (the sum of the lengths of the $c_{ij}$) is in $O(n^2)$.

\item For each tuple $\mathbf c$,  run Dehn's algorithm to check $c_{j1}c_{j2}c_{j3}=_G1$ for $1\leq j\leq q$. If this holds for all $j$, write down a system of $3q$ equations $$X_j=P_jc_{j1}Q_j, Y_j=Q_j^{-1}c_{j2}R_j, Z_j=P_jc_{j3}R_j.$$
Note that the resulting system, $\Phi_{\mathbf c}$, has size in $O(n^2)$.
\item We now call  the algorithm of the authors and Diekert \cite{CDE} to find all solutions to $\Phi_{\mathbf c}$
 in the free group generated by $S$.
This algorithm, on input of size $O(n^2)$, runs in $\NSPACE(n^2\log n)$, and prints a description of the EDT0L grammar which generates all tuples of solutions  as reduced words in $S^*$.  Specifically it prints nodes and edges of a trim NFA which is the rational control for the EDT0L grammar (see Definition~\ref{def:et0lasfeld} below).
Modify the algorithm so that the  nodes printed
include the label $\mathbf c$ which has length $O(n^2)$ (so does not affect the complexity). 
\item Delete the current system stored, and move to the next tuple $\mathbf c$.
\item 
At the end, print out a new start node and  $\epsilon$ edges to the start node of the NFA for the system $\Phi_{\mathbf c}$ for all $\mathbf c$ already printed. 
\end{enumerate}

The  NFA that is printed gives an EDT0L grammar that generates a language of tuples which is a  covering solution to the original system in the hyperbolic group.
To obtain the full set of solutions as shortlex geodesic words we need to perform further steps.
 Using the facts that %
 canonical representatives are $(\lambda,\mu)$-quasigeodesics, and
\begin{itemize}\item 
the full set of $(\lambda,\mu)$-quasigeodesics, $Q_{S,\lambda,\mu}$\item  the set of all pairs $\{(u,v)\in Q_{S,\lambda,\mu}\mid u=_G v\}$ 
\item the set of all shortlex geodesics in $G$ 
\end{itemize} are all regular,
we can obtain from the covering solution an ET0L language, in the same space complexity (by Proposition~\ref{prop:closureET0L} below), which represents the full set of solutions in shortlex geodesic words. 
Then finally, because of the special form of our solutions, we can apply a version of the {\em Copying Lemma} of Ehrenfeucht and Rozenberg \cite{EhrenRozenCopyingEDT0L} to show that in fact the resulting language of shortlex representatives is EDT0L  in $\NSPACE(n^2\log n)$.

Details for handling the case of hyperbolic groups with torsion also follows this general scheme, however finding the analogue of canonical 
representatives is harder in this case, so further work is required, and we describe this in Section~\ref{sec:torsion}.

\section{E(D)T0L in \PSPACE}\label{sec:EDT0L}

\subsection{ET0L and EDT0L languages}

Let $C$ be an alphabet. 
A \emph{table} for $C$ is a finite subset of $C\times C^*$.
If $(c,v)$ is in some table $t$, we say that $(c,v)$ is a \emph{rule} for $c$.
A table $t$ is {\em deterministic} if for each $c\in C$ there is exactly one $v\in C^*$ with  $(c,v)\in t$.

If $t$ is a table and $u\in C^*$ then  we write 
$u\longrightarrow^t v$ to mean that $v$ is obtained by applying rules from $t$ to each letter of $u$. That is, $u=a_1\dots a_n$,  $a_i\in C$,  $v=v_1\dots v_n$, $v_i\in C^*$, and $(a_i,v_i)\in t$ for $1\leq i\leq n$.
If $H$ is a set of tables  and $r\in H^*$ then we write $u\longrightarrow^{r} v$ to mean that there is a sequence of words $u=v_0,v_1,\dots, v_n=v\in C^*$ 
such that $v_{i-1}\longrightarrow^{t_i} v_i$ for $1\leq i\leq n$ where $r=t_1\dots t_n$. If $R\subseteq H^*$ we write $u\longrightarrow^{R} v$ if $u\longrightarrow^{r} v$ for some $r\in R$.

\begin{definition}[\cite{Asveld}]\label{def:et0lasfeld}
	Let $\Sigma$ be an alphabet. We say that $L\subseteq \Sigma^*$ is an {\em ET0L} language if there is an alphabet $C$ with $\Sigma\subseteq C$, a finite set $H\subset \mathscr P(C\times C^*)$ of tables, 
	a regular language $R \subseteq H^*$ and a letter $c_0\in C$ such that
\begin{displaymath}
	L = \{ w \in \Sigma^* \mid c_0 \longrightarrow^R w\}. 
\end{displaymath}
In the case when every table $h \in R$ is deterministic, i.e. each $h \in R$ is in fact a homomorphism, we write $	L = \{ r(c_0) \in \Sigma^* \mid r\in R \}$
 and say that $L$ is {\em EDT0L}.
The set ${R}$ is called the {\em rational control}, the symbol $c_0$ the \emph{start symbol} and $C$ the {\em extended alphabet}.
\end{definition}

\subsection{Space complexity for E(D)T0L}

Let  $f\colon \N\to\N$ be a function.
Recall an algorithm is said to run in \NSPACE$(f(n))$ if it can be performed by a non-deterministic Turing machine with a read-only input tape, a write-only output tape, and a read-write work tape, with the work tape restricted to using $O(f(n))$ squares on input of size $n$.  We use the notation $L(\mathcal A)$ to denote the language  accepted by the automaton $\mathcal A$. The following definition formalises the idea of producing some E(D)T0L language (such as the solution set of some system of equations) in \NSPACE$(f(n))$, where the language is the output of a computation with input (such as a system of equations) of size $n$.

\begin{definition} Let 
 $\Sigma$ be a  (fixed) alphabet and  $f\colon \N\to\N$ a function. If there is an $\NSPACE(f(n))$ algorithm that on input $\Omega$ of size $n$ outputs the specification of an ET0L language $L_{\Omega}\subseteq \Sigma^*$, then we say that $L_{\Omega}$ is {\em ET0L in $\NSPACE(f(n))$}. 
 
 Here the specification of $L_{\Omega}$ consists of:
  \begin{itemize}\item
 an extended alphabet $C\supseteq \Sigma$,  
 \item a start symbol $c_0\in C$, 
 \item a finite list 
of nodes of a (trim) NFA $\mathcal A$, labeled by some data, some possibly marked as initial and/or final, 
\item  a finite list $\{(u,v,h)\}$  of edges of 
$\mathcal A$ where $u,v$ are nodes and $h\in \mathscr P(C\times C^*)$ is a 
 table
\end{itemize}
such that $L_{\Omega}=\{w\in \Sigma^*\mid c_0\to^{L(\mathcal A)} w\}$.

A language $L_{\Omega}$ is  {\em EDT0L 
 in
 $\NSPACE(f(n))$} if, in addition, every table $h$  labelling  an edge of $\mathcal A$ is deterministic.
 \end{definition}

 Note that the entire  print-out is not required to be in $O(f(n))$ space. 
Previous results of the authors with Diekert can now be restated as follows.
\begin{theorem}[{\cite[Theorem 2.1]{CDE}}] The set of all solutions to a system   of size $n$ of equations (with rational constraints),
as reduced words, in a free group is EDT0L in $\NSPACE(n\log n)$.
\end{theorem}
\begin{theorem}[{\cite[Theorem 45]{DEarxiv}}] The set of all solutions to a system of size $n$ of equations  (with  rational constraints),  
as  words in a particular quasigeodesic normal form over a certain finite generating set, in a virtually free group is EDT0L in $\NSPACE(n^2\log n)$. 
\end{theorem}

\begin{remark} In our  applications below we have $\Omega$ representing  some system of equations and inequations, with $|\Omega|=n$, and we construct algorithms where the extended alphabet $C$ has size $|C|\in O(n)$ in the torsion-free case and $|C|\in O(n^2)$ in the torsion case. This means  we can write down the entire alphabet $C$ as binary strings within our space bounds.  Moreover, each element $(c,v)$ of any table we construct has $v$ of (fixed) bounded length, so we can write down entire tables within our space bounds.
 \end{remark}

\subsection{Closure properties}

It is well known \cite[Theorem 2.8]{LsysHandbook}  that ET0L is  a full AFL  (closed under homomorphism, inverse homomorphism, finite union, intersection with regular languages).
Here we show the space complexity of an ET0L language is not affected by these operations.  

 \begin{proposition}\label{prop:closureET0L}
 Let $\Sigma, \Gamma$ be finite alphabets of fixed size, $M$ an NFA of constant size with $L(M) \subseteq \Sigma^*$,  and $\varphi\colon \Gamma^*\to \Sigma^*$, $\psi\colon \Sigma^*\to \Gamma^*$   homomorphisms. 
 If $L_{\Omega_1}, L_{\Omega_2}\subseteq \Sigma^*$ are E(D)T0L in \NSPACE$(f(n))$ (on inputs $\Omega_1, \Omega_2$, respectively, with $|\Omega_1|, |\Omega_2|\in O(n)$)  
 then 
 \begin{itemize}
 \item (homomorphism) $\psi(L_{\Omega_1})$ is E(D)T0L  in \NSPACE$(f(n))$,
 \item (intersection with regular) $L_{\Omega_1}\cap L(M)$ is E(D)T0L  in \NSPACE$(f(n))$,
    \item (union) $L_{\Omega_1}\cup L_{\Omega_2}$ is E(D)T0L  in \NSPACE$(f(n))$,
   \item (inverse homomorphism) $\varphi^{-1}(L_{\Omega_1})$ is ET0L in \NSPACE$(f(n))$.

 \end{itemize}\end{proposition}

The proof is straightforward keeping track of complexity in the standard proofs \cite{AsveldChar, Culik}. Note EDT0L is not closed under inverse homomorphism \cite{EhrenRozenInverseHomomEDT0L}.
 \begin{proposition}[Projection onto a factor]\label{prop:projection}
If $L_{\Omega} \subseteq \Sigma^*$ is E(D)T0L in \NSPACE$(f(n))$ on an input $\Omega$ of size $n$, and for some   fixed integer $s$
all words in $L_{\Omega}$ have the form $u_1\#\dots \#u_s$ with $u_i\in\left(\Sigma\setminus \{\#\}\right)^*$, and $1\leq i\leq j\leq s$, then 
 $$L=\{u_i\#\dots\#u_j\mid u_1\#\dots \#u_i\#\dots\#u_j\#\dots \#u_s\in L_{\Omega}\}$$  is E(D)T0L  in \NSPACE$(f(n))$.  \end{proposition}

\subsection{From ET0L to EDT0L}

In computing the full solution set to equations as shortlex geodesic words, we will need to take inverse homomorphism.
Even though in general the image under an inverse homomorphism of an EDT0L language is just ET0L, because of the special structure of solution sets we can apply the {\em  Copying Lemma} of Ehrenfeucht and Rozenberg \cite{EhrenRozenCopyingEDT0L} to show the following.

\begin{proposition}\label{prop:copyME}
Let $S$ be an alphabet and  $h:S\to S'$ be a homomorphism of from $S$ to a disjoint alphabet $S'=\{s'\mid s\in S\}$ defined by $h(s)=s'$.  Let $\sep$ be a symbol not in $S\cup S'$ and define $h(\sep)=\sep$.
Let $L_1$ be a set of words of the form $w\sep h(w)$ where $w\in S^*$.
 If $L_1$ is ET0L in $\NSPACE(f(n))$, then 
$L_2=\{w\mid w\sep h(w)\in L_1\}$ is EDT0L in $\NSPACE(f(n))$.

\end{proposition}\begin{proof}
By \cite{EhrenRozenCopyingEDT0L}, any nondeterministic table in the grammar for $L_1$ can be replaced by a finite number of deterministic tables (essentially, if nondeterminism allowed some letter $c\in C$ to produce two different results, then  some word in $L_1$ would not have the form $w\sep h(w)$). So without loss of generality we can replace a table $f$  containing $(c,v_1),\dots, (c,v_k)$ by $k$ tables $f_i$ containing $(c,v_i)$ only). 
This modification is clearly in the same space bound. Project onto the prefix using Proposition~\ref{prop:projection}.
\end{proof}

 \section{Hyperbolic groups}\label{sec:hyp-intro}

\subsection{Definitions}
Recall the {\em Cayley graph} for a group $G$ with respect to a finite symmetric generating set $S$ is a directed graph $\Gamma(G,S)$ with vertices labeled by $g\in G$ and a directed edge $(g,h)$ labeled by $s\in S$ whenever $h=_Ggs$.
Let $\ell(p)$, $i(p)$ and $f(p)$ resp. be the length, initial and terminal vertices of a path $p$ in the Cayley graph. 
A path  $p$ is {\it geodesic} if 
$\ell(p)$ is minimal among the lengths of all paths $q$ with the same endpoints. 
If $x,y$ are two points in  $\Gamma(G,S)$, we define $d(x,y)$ to be the length of a shortest path from $x$ to $y$ in  $\Gamma(G,S)$.

 \begin{definition}[$\delta$-hyperbolic group (Gromov)]
 Let $G$ be a group with finite symmetric generating set $S$, and let $\delta\geq 0$ be a fixed real number. 
 If  $p,q,r$ are geodesic paths in $\Gamma(G,S)$ with $f(p)=i(q),f(q)=i(r),f(r)=i(p)$, we call $[p,q,r]$ a 
{\em geodesic triangle}. 
A geodesic triangle is {\em $\delta$-slim} if $p$ is contained in a $\delta$-neighbourhood of $q\cup r$, that is, every point on one side of the triangle is within $\delta$ of some point on one of the other sides. (See for example Figure~\ref{subfig:geod}.)
We say $(G,S)$ is {\em $\delta$-hyperbolic} if every geodesic triangle in $\Gamma(G,S)$ is $\delta$-slim. We say $(G,S)$ is {\em hyperbolic} if it is $\delta$-hyperbolic for some $\delta\geq 0$. 
 \end{definition}
 
It is a straightforward to show that being hyperbolic is independent of choice of finite generating set.  Thus we say $G$ is hyperbolic if $(G,S)$ is for some finite generating set $S$.

 \begin{lemma}[Dehn presentation]\label{lem:Dehn}
 $G$ is hyperbolic if and only if 
 there is a finite list of pairs of words $(u_i,v_i)\in S^*\times S^*$ with $|u_i|>|v_i|$ and $u_i=_G v_i$ such that the following holds: if $w\in S^*$ is equal to the identity of $G$ then it contains some $u_i$ as a factor.
 
 \end{lemma}
 This gives  an algorithm to decide whether or not a word $w\in S^*$ is equal to the identity: while $\ell(w)>0$, look for some $u_i$ factor.
 If there is none, then $w\neq_G 1$. Else replace $u_i$ by $v_i$ (which is shorter). This procedure is called {\em Dehn's algorithm}.
 
  \begin{lemma}
 Dehn's algorithm runs in (linear time and) linear space.\end{lemma}

\begin{definition}[Quasigeodesic]\label{defn:qg}
For $\lambda \geq 1, \mu\geq 0$ real numbers, a path $p$ in $\Gamma(G,S)$ is a $(\lambda, \mu)$-{\em quasigeodesic} 
if for any subpath $q$ of $p$ we have 
$\ell(q)\leq \lambda d(i(q),f(q))+\mu$. 
\end{definition}

  Throughout this article, we assume $G$ is a fixed hyperbolic group which we treat as a constant for complexity purposes. We assume we are given $(G,S)$, the constant $\delta$, the finite list of pairs for Dehn's algorithm, and any other constants depending only on the group, for example the constants $\lambda,\mu$ in 
  Prop.~\ref{canrep_qg}  below.

\subsection{Languages in hyperbolic groups}

\begin{proposition}\label{prop:reg-hyp}

Let $G$ be a fixed hyperbolic group with finite  generating set $S$, $\lambda\geq 1, \mu\geq 0$ constants with $\lambda\in\mathbb Q$ and $\mu$ sufficiently large. 
Then the following sets are regular languages.
\begin{enumerate}
\item
The set of all geodesics over $S$.
\item
The set of all shortlex geodesics over $S$. 
\item The set of all $(\lambda,\mu)$-quasigeodesics, $Q_{S,\lambda, \mu}\subseteq S^*$. 
\end{enumerate}
Furthermore, the set of all pairs of words $(u,v)\in Q_{S,\lambda, \mu}^2$ such that $u=_Gv$ is accepted by an asynchronous 2-tape automaton.
\end{proposition}
See  \cite{WordProc,HoltRees}.

\subsection{Main reduction result}
Here is our  key technical result. 

\begin{proposition}[Covering  to full solution sets]\label{prop:shortlex-hyp}
Let $G$ be a hyperbolic group with finite symmetric  generating set $S$. Let $h:S\to S'$ and $Q_{S,\lambda,\mu}$ be as defined above, $\#,\sep$  symbols not in $S\cup S'$, $h(\#)=\#,h(\sep)=\sep$, 
and $\mathcal T\subseteq Q_{S,\lambda, \mu}$ 
a regular set of quasigeodesic words in bijection with $G$.
Suppose $L_1\subseteq (S\cup S'\cup \{\#,\sep\})^*$ consists of words of the form \[u_1\#\dots\#u_r\sep h(v_1)\#\dots \#h(v_r), \ u_i,v_i\in Q_{S,\lambda, \mu}, u_i=_Gv_i, 1\leq i\leq r.\] If $L_1$ 
 is ET0L in $\NSPACE(f(n))$,  then
 \begin{enumerate}\item
  $L_{Q}=\{w_1\#\dots \#w_r\sep h(z_1)\#\dots \#h(z_r)\mid \exists u_1\#\dots\#u_r\sep h(v_1)\#\dots \#h(v_r) \in L_1, w_i=_Gz_i=_Gu_i, w_i,z_i\in Q_{S,\lambda,\mu} \}$    is ET0L  in $\NSPACE(f(n))$.
  
  \item 
  $L_{\mathcal T}=\{w_1\#\dots \#w_r\mid \exists u_1\#\dots\#u_r\sep h(v_1)\#\dots \#h(v_r) \in L_1, w_i=_Gu_i, w_i\in \mathcal T\}$
  is EDT0L  in $\NSPACE(f(n))$.\end{enumerate}
\end{proposition}
  The proof involves a series of operations as in 
  Proposition~\ref{prop:closureET0L}--\ref{prop:copyME}. 
  Note that the set of all shortlex geodesics is a suitable choice for $\mathcal T$ in the proposition.

\bigskip

\section{Reduction from torsion-free hyperbolic to free groups}\label{sec:torsionfree}

Section \ref{sec:overview} contains an overview of the general algorithm for solving equations in torsion-free hyperbolic groups. Here we provide further details, and give a proof of the soundness and completeness of our algorithm. The algorithm relies on the existence and special properties of canonical representatives, 
whose construction is very technical (see Appendix \ref{appendix:torsionfree}). 
Their existence guarantees that the solutions of any equation in a torsion-free hyperbolic group can be found by solving an associated system in the free group on $S$, while the fact that they are quasigeodesics (see \ref{canrep_qg})
 allows us to apply the results of the previous sections to obtain the EDT0L characterisation of solutions in shortlex normal form.

\begin{proposition}\label{prop:MAIN-torsionfree}
Let $G$ be a torsion-free hyperbolic group, with  finite symmetric generating set $S$.   Let $ \Phi$ be a system of equations and inequations of input size $n$ as in Section~\ref{sec:notationSolns}. Let $h:S\to S',\#,\sep$  be as in Proposition~\ref{prop:shortlex-hyp}.
Then there exist $\lambda\geq 1, \mu\geq 0 $ and 
\[L=\{w_1\#\dots\# w_m\sep h(w_1) \#\dots \# h(w_m) \mid w_i \in Q_{S,\lambda,\mu}, 1\leq i \leq m\} \]
 such that $\{w\mid  w \sep  h(w)\in L\}$ is a covering solution for $\Phi$, and $L$ is  EDT0L  in $\NSPACE(n^2\log n)$.
\end{proposition}
Applying Proposition~\ref{prop:shortlex-hyp} immediately gives  Theorem~\ref{thmTorsionFree}.
\begin{proof}
We produce a language $L$ of quasigeodesic words over $S$ such that the projection of any tuple in $L$ is in the group element solution set $\text{Sol}_G(\Phi)$ (soundness). We then prove 
(using \cite[Corollary 4.4]{RS95})  that any solution in $\text{Sol}_G(\Phi)$ is the projection of some tuple in $L$ (completeness).  Our proof  follows the outline presented in Section~\ref{sec:overview}.

\paragraph*{1. Preprocessing} \begin{itemize}\item
(Remove inequations) We first transform $\Phi$ into a system consisting entirely of equations by adding a variable $x_D$ to $\mathcal{X}$ and replacing any inequation $\varphi_j(\mathcal X, \mathcal{A})\neq1$ by $\varphi_j(\mathcal X, \mathcal{A}) = x_D$, with the constraint $x_D\neq_G 1$.

\item (Triangulation) 
We transform each equation into several equations of length $3$, by introducing new variables. This can always be done (see the discussion in \cite[Section 4]{CDE}), and it produces approximately $\sum_{i=1}^s l_i\in O(n)$ triangular equations with set of variables $\mathcal{Z}$
 where $m \leq |\mathcal Z|\in O(n)$ and $\mathcal X \subset \mathcal{Z}$. 
From now on assume that the system $\Phi$ consists of $q\in O(n)$ equations of the form
 $X_jY_j=Z_j$ where $1\leq j \leq q$. 
\end{itemize}
\paragraph*{2. Lifting $\Phi$ to the free group on $S$}

In \cite[Theorem 4.2]{RS95} Rips and Sela define a constant, which they call $`bp'$, that roughly bounds the circumference of the `centres' of the triangles whose edges are canonical representatives.
We denote here $bp$ by $\rho$, and note that $\rho\in O(q)= O(n)$ depends on $\delta$ and linearly on $q$. As described in Section~\ref{sec:overview} we run in lex order through all possible tuples of words $\mathbf c=(c_{11},c_{12},c_{13},\dots, c_{q1},c_{q2},c_{q3})$ with $c_{ji}\in S^*,\ell(c_{ji})\leq \rho\in O(n)$.
 For each tuple $\mathbf c$ we use Dehn's algorithm to check $c_{j1}c_{j2}c_{j3}=_G1$, and if this holds for all $1\leq j\leq q$ we then construct a system $\Phi_{\mathbf c}$ of equations of the form 
\begin{equation}\label{csystem}
X_j=P_jc_{j1}Q_j, \ Y_j=Q_j^{-1}c_{j2}R_j, \ Z_j=P_jc_{j3}R_j, \ \ 1\leq j\leq q,
\end{equation}
which has size $O(n^2)$.  In order to avoid an exponential size complexity we write down each system $\Phi_{\mathbf c}$ one at a time, so the space required for this step is $O(n^2)$.  Let $\mathcal{Y}\supset\mathcal{Z}\supset \mathcal X$ be the new set of variables.

\paragraph*{3. Some observations} 
We pause to make the following observations. Any solution to  $\Phi_{\mathbf c}$ 
in the free group $F(S)$ is guaranteed to be a solution to $\Phi$ in the original hyperbolic group $G$. Thus if $S_1\subseteq F(S)^m$ is  a group element solution to $\Phi_{\mathbf c}$ then $\pi(S_1)$  is a group element solution to $\Phi$ in $G$. This will show soundness below.

Secondly, if $(g_1,\dots, g_m)\in  G^m$ is a solution to $\Phi$ in the original hyperbolic group,  \cite[Theorem 4.2 and Corollary 4.4]{RS95} (see Theorem 31 
in Appendix~\ref{appendix:torsionfree}) 
guarantees that there exist canonical representatives $w_i\in Q_{S,\lambda,\mu}$ with $w_i=_Gg_i$ for $1\leq i\leq m$, which have reduced forms $u_i=_Gw_i$ for $1\leq i\leq m$, and our construction is guaranteed to capture any such collection of words. This will show completeness below.

Thirdly, note that the constraint that a word $w\in S^*$ must be a $(\lambda,\mu)$-quasigeodesic and satisfy $w=_G1$ implies that $\ell(w)\leq \mu$. Therefore we can construct a DFA $\mathcal D$ which accepts all words in $S^*$ equal to $1$ in the hyperbolic group $G$ of length at most $\mu$  in constant space (using for example Dehn's algorithm). 
In our next step, we will use this rational constraint to handle
the variable $x_D$ added in the first step above (to remove inequalities). %

Now let us complete the construction by finding the covering solution required.

\paragraph*{4. Covering solution set} 
We now run  the algorithm from \cite{CDE} (which we will refer to as the CDE algorithm) which takes input $\Phi_\mathbf{c}$, which has size in $O(n^2)$, plus the  rational constraint $x_D\not\in  L(\mathcal D)$, plus for each   $y\in \mathcal Y$ the rational constraint that the solution for  $y$  is a word in $Q_{S,\lambda, \mu}$. Since these constraints have constant size (depending only on the group $G$, not the system $\Phi$),  they do not contribute to the $O(n^2)$ size of the input to the CDE algorithm. 

We make two modifications to the details of the CDE algorithm. First, every node printed by the algorithm should include the additional label $\mathbf c$. (This ensures the NFA we print for each system $\Phi_{\mathbf c}$ is distinct.)
This does not affect the complexity since $\mathbf c$ has size in $O(n^2)$. 

Second, so that we can apply Proposition~\ref{prop:copyME} later, we modify the form of `extended equations' in \cite{CDE} by inserting the factor $\sep h(W)$  in the appropriate position(s). This simply increases the size of the nodes by a factor (of two). 

We run the CDE  algorithm to print an NFA (possibly empty) for each $\Phi_{\mathbf c}$, which is the rational control for an EDT0L grammar that produces all solutions as freely reduced words for elements of $F(S)$ which correspond to solutions as $(\lambda,\mu)$-quasigeodesics to the same system $\Phi_{\mathbf c}$ in the hyperbolic group. 
If $(w_1,\dots, w_m)$ is a solution in canonical representatives to $\Phi$ then $(u_1,\dots, u_m,\dots u_{|\mathcal Y|})$  will be included in the  solution to  $ \Phi_{\mathbf c}$ output by the CDE algorithm, with $u_i$ the reduced forms of $w_i$ for $1\leq i\leq m$. This shows completeness once we union the grammars from all systems $\Phi_{\mathbf c}$ together.


Adding a new start node with edges to each of the start nodes of the NFA's with label $\mathbf c$, we obtain a rational control for the EDT0L grammar generating $L$ as required.
The space required is exactly that required by the CDE algorithm on input $O(n^2)$, which is $\NSPACE(n^2\log n)$.
 \end{proof}

\section{Reduction from  hyperbolic with torsion to virtually free groups}\label{sec:torsion}

In the case of a hyperbolic group $G$  with torsion, the general approach of Rips and Sela can still be applied, but the existence of canonical representatives is not always guaranteed (see Delzant \cite[Rem.III.1]{Delzant}). To get around this, Dahmani and Guirardel `fatten' the Cayley graph $\Gamma(G,S)$ of $G$ to a larger graph $\mathcal K$ which contains $\Gamma(G,S)$ (in fact $\Gamma(G,S)$ with midpoints of edges included), and solve equations in $G$ by considering equalities of paths in  $\mathcal{K}$. 
 More precisely, $\mathcal{K}$ is the $1$-skeleton of the barycentric subdivision of a Rips complex of $G$ (see Appendix~\ref{appendix:torsion}
  for definitions).

\begin{definition} \label{Kpaths}
Let $\gamma, \gamma'$ be paths in $\mathcal{K}$.
\begin{itemize}
\item[(i)] We denote by $i(\gamma)$ the initial vertex of $\gamma$, by $f(\gamma)$ the final vertex of $\gamma$, and by $\overline{\gamma}$ the reverse of $\gamma$ starting at $f(\gamma)$ and ending at $i(\gamma)$.  
\item[(ii)] We say that $\gamma$ is \emph{reduced} if it contains no backtracking, that is, no subpath of length $2$ of the form $e \overline{e}$. 
\item[(iii)] We write $\gamma \gamma'$ for the concatenation of $\gamma$, $\gamma'$ if $i(\gamma')=f(\gamma)$. 
\item[(iv)] Two paths in $\mathcal{K}$ are \textit{homotopic} if one can obtain a path from the other by adding or deleting backtracking subpaths. Each homotopy class has a unique reduced representative. 
\end{itemize}
\end{definition}

Let $V$ be the set of all homotopy classes $[\gamma]$ of paths $\gamma$ in $\mathcal{K}$ with $i(\gamma)=1_G$, and $f(\gamma)\in G$. For $[\gamma]$, $[\gamma'] \in V$ define their product $[\gamma][\gamma']=[\gamma^v\gamma']$, where $\gamma ^v\gamma'$ denotes the concatenation of $\gamma$ and the translate $^v\gamma'$ of $\gamma'$ by $v=f(\gamma)$, and let $[\gamma]^{-1}$ be the homotopy class of $^{v^{-1}}\overline{\gamma}$. Then $V$ is a group that projects onto $G$ by the final vertex map $f$, that is,
 $f:V \twoheadrightarrow G$ is a surjective homomorphism. Moreover, since $G$ has an action on $\mathcal{K}$ induced by the natural action on its Rips complex, $V$ will act on $\mathcal{K}$ as well. This gives rise to an action of $V$ onto the universal cover $T$ (which is a tree) of $\mathcal{K}$, and \cite[Lemma 9.9]{DG} shows that the quotient $T/V$ is a finite graph (isomorphic to $\mathcal{K}/G$) of finite groups, and so $V$ is virtually free.
 
 We assume that the algorithmic construction (see \cite[Lemma 9.9]{DG}) of a presentation for $V$ is part of the preprocessing of the algorithm, will be treated as a constant, and will not be included in the complexity discussion.

The first step in solving a system $\Phi$ of equations in $G$ is to translate $\Phi$ into identities between quasigeodesic paths (with start and end point in $G$) in $\mathcal{K}$, defined as $\mathcal{Q}\mathcal{G}_{\lambda_1, \mu_1}(V)$ in Equation~(\ref{qgK})
 paths which can be seen as the analogues of the canonical representatives from the torsion-free case. This can be done by Proposition 9.8 \cite{DG}. (see Proposition \ref{prop:K}). 
The second step in solving $\Phi$ is to express the equalities of quasigeodesic paths in $\mathcal{K}$ in terms of equations in the virtually free group $V$ based on $\mathcal{K}$. Finally, Proposition 9.10 \cite{DG} 
shows it is sufficient to solve the systems of equations in $V$ in order to obtain the solutions of the system $\Phi$ in $G$. 

In the virtually free group $V$ we will use the results from \cite{DEarxiv}. Let $Y$  be the generating set of $V$ and $T\subseteq Y^*$ the set of normal forms for $V$ over $Y$ as in  \cite[Remark 44, page 50]{DEarxiv},
and let $$\text{Sol}_{T,V}(\Psi)=\{w_1\#\dots\# w_m )\in T^n\mid (\pi(w_1),\dots,\pi(w_m) )\text{ solves } \Psi \text{ in } V\}$$ be the language of $T$-solutions in $V$ of a system $\Psi$ of size $|\Psi|=O(k)$; by \cite{DEarxiv} the language $\text{Sol}_{T,V}(\Psi)$ consists of ($\lambda_Y, \mu_Y$)-quasigeodesics and is EDT0L in $\NSPACE(k^2\log k)$ 
over $Y$.



\begin{proposition}\label{prop:MAIN-torsion} 
Let $G$ be a  hyperbolic group with torsion, with  finite symmetric generating set $S$.   Let $ \Phi$ be a system of equations and inequations with $|\Phi|=n$ as in Section~\ref{sec:notationSolns}. Let $h:S\to S',\#,\sep$  be as in Proposition~\ref{prop:shortlex-hyp}.
Then there exist $\lambda\geq 1, \mu\geq 0 $ and 
\[L=\{w_1\#\dots\# w_m\sep h(v_1) \#\dots \# h(v_m) \mid w_i, v_i\in Q_{S,\lambda,\mu}, w_i=_Gv_i, 1\leq i \leq m\} \]
 such that $\{w\mid  w\sep h(v)\in L\}$ is a covering solution for $\Phi$, and $L$ is  ET0L  in $\NSPACE(n^4\log n)$.
\end{proposition}
Again applying Proposition~\ref{prop:shortlex-hyp}  immediately gives Theorem~\ref{thmTorsion}.

%
%
%

Before proving Proposition~\ref{prop:MAIN-torsion} we need to show how one can translate between elements and words in $V$ over the generating set $Y$, and elements and words in $G$ over $S$ via the graph $\mathcal{K}$, so that the EDT0L characterisation of languages is preserved.
\begin{notation} \label{notationZ}
Let $Z$ be some generating set of $V$ and let $\pi\colon Z^* \to V$ be the standard projection map from words to group elements in $V$. 
\begin{enumerate}
\item[(i)] For each $z_i \in Z$ there exists a unique reduced path $p_i$ in $\mathcal{K}$ with $i(p_i)=1_G$ and $f(p_i) \in G$; by concatenation 
 for each word $w=z_{i_1} \dots z_{i_k}$ over $Z$ there is then a unique path denoted 
 \vspace{-.5cm}
\begin{equation}\label{unreduced}
p_w=p_{i_1} \dots p_{i_k}
\end{equation}
 with $i(p_w)=1_{\mathcal{K}}=1_G$ and $f(p_w) \in G$. 

\item[(ii)] For each $z_i\in Z$, assign a geodesic path $\gamma_i$ in the Cayley graph $\Gamma(G,S)$ such that  $i(\gamma_i)=1_G$ and $f(\gamma_i)=f(p_i) \in G$, where $p_i$ as in (i). Let $\sigma\colon Z^* \to S^*$ be the map/substitution given by $\sigma(z_i)=\gamma_i$; by concatenation one can associate to each word $w=z_{i_1} \dots z_{i_k}$ over $Z$ a path in $\Gamma(G,S)$ denoted 
\begin{equation}\label{Gpath}
\gamma_w=\gamma_{i_1} \dots \gamma_{i_k}=\sigma(w)
\end{equation}
 with $i(p_w)=1_G$ and $f(\gamma_w)=f(w) \in G$. 

\item[(iii)] There exists a unique reduced path, denoted $p_{\pi(w)}$, which is homotopic to $p_w$.

\end{enumerate}
\end{notation}


\begin{proof*}{Proof of Proposition~\ref{prop:MAIN-torsion}}
The algorithm to produce the language of solutions for $\Phi$ is  similar to that outlined in Section \ref{sec:overview} and detailed in the proof
of Proposition \ref{prop:MAIN-torsionfree}, 
but it applies to different groups. The triangulation of $\Phi$ and introduction of a variable with rational constraint to deal with the inequations proceeds in the same manner. 
Again, we 
suppose after preprocessing we have $q\in O(n)$ triangular equations.

Then for $\kappa\in O(n)$ as in Proposition~\ref{prop:V}.
define $V_{\leq \kappa}=\{[\gamma] \in V\mid \gamma \textrm{\ reduced\ and\ } \ell_{\mathcal{K}}(\gamma)\leq \kappa \}.$ One lifts the system $\Phi$ in $G$ to a finite set of systems $\Psi_{\mathbf c}$ 
 in the virtually free group $V$, one system for each $q$-tuple $\mathbf{c}$ of triples $(c_1, c_2, c_3)$ with $c_i \in V_{\leq \kappa}$ and such that $f(c_1c_2c_3)=1_G$, as in 
 Proposition \ref{prop:V}.
  We enumerate these tuples by enumerating triples of words $(v_1,v_2,v_3)$ over the generating set $Y$ of $V$ with $\ell_Y(v_i)\leq \kappa_Y$, where $\kappa_Y\in O(q)$ is a constant depending on $\kappa$, as in Lemma \ref{kappaV}(ii). 
 By Lemma \ref{kappaV}(ii) the tuples of path triples $(p_{v_1},p_{v_2},p_{v_3})$ (see (\ref{unreduced})) in $\mathcal{K}$ contain all $q$-tuples of triples $(c_1, c_2, c_3)$ with $c_i \in V_{\leq \kappa_{Y}}$, 
 up to homotopy. Then for each triple $(v_1,v_2,v_3)$ we check whether $f(v_1v_2v_3)=1_G$, and this is done by checking whether $\sigma(v_1)\sigma(v_2)\sigma(v_3)=_G 1$ using the Dehn algorithm in $G$.

Then each system $\Psi_{\mathbf{c}}$ is obtained as in (\ref{csystem}) in the proof of Proposition \ref{prop:MAIN-torsionfree} and has input size $O(q^2)\in O(n^2)$ since it has $O(q)$ equations, each of length in $O(q)$, and the factors $c_{i}$ inserted also have length  in $O(q)$. For each system $\Psi_{\mathbf{c}}$ over $V$ we apply the algorithm in \cite{DEarxiv} and obtain the set of solutions $\text{Sol}_{T,V}(\Psi_{\mathbf{c}})$ as an EDT0L in $\NSPACE((q^2)^2\log (q^2))= \NSPACE(n^4\log n)$ of $(\lambda_Y,\mu_Y)$-quasigeodesics over $Y$. 

Now let $\mathcal{Q}\text{Sol}_{T,V}(\Psi_{\mathbf{c}})$ be the set of all $(\lambda'_1,\mu'_1)$-quasigeodesics which represent solutions of $\Psi_{\mathbf{c}}$ in $V$ over $Y$. By Proposition \ref{prop:shortlex-hyp} this language is ET0L and by Corollary \ref{cor:KtoY} it contains at least one word over $Y$ for each solution in $\mathcal{Q}\mathcal{G}_{\lambda_1, \mu_1}(V)$. 

Then $\sigma(\mathcal{Q}\text{Sol}_{T,V}(\Psi_{\mathbf c}))$ is ET0L since ET0L languages are preserved by substitutions, and by Prop.
\ref{prop:V}
 $\mathcal{S}=\cup_{\mathbf{c}} \sigma(\mathcal{Q}\text{Sol}_{T,V}(\Psi_{\mathbf{c}}))$ contains $\text{Sol}_{G}(\Phi)$,
 so it is a covering solution set of $\Phi$. By Lemma \ref{VtoG} the set $\mathcal{S}$ consists of at least one $(\lambda_G, \mu_G$)-quasigeodesic over $S$ for each solution, and then by intersection with the regular set  $Q_{S,\lambda_G, \mu_G}$ of quasigeodesics in $G$ over $S$ we obtain a set of solutions for $\Phi$ consisting of ($\lambda_G, \mu_G$)-quasigeodesics.

Finally, we run the modified DE algorithm (inserting the additional $\sep h(W)$ and label $\mathbf c$ for each node printed) to print an NFA for each $\Phi_{\mathbf c}$ for the EDT0L grammar for $\text{Sol}_{T,V}(\Psi_{\mathbf{c}})$, 
which we union using an extra start node as before.
From the above work this grammar generates the language $L$ as required.
\end{proof*}

\begin{lemma}\label{kappaV}
\begin{enumerate}
\item[(i)] If $c \in V$ and the reduced path representing $c$ in $\mathcal{K}$ is an $(a,b)$-quasigeodesic, then there exists a word $w$ on $Y$ representing $c$ such that $w$ is an $(a',b')$-quasigeodesic, where $a',b'$ depend on $a,b$ and $Y$.
\item[(ii)] If $c \in V$ and the length of the reduced path representing $c$ in $\mathcal{K}$ is $\leq L$,
 then there exists a word $w$ on $Y$ representing $c$ such that $\ell_Y(w)\leq L_Y$, where $L_Y$ depends on $L$ and $Y$.  
\end{enumerate}
\end{lemma}
\begin{corollary}\label{cor:KtoY}
For any element $v \in \mathcal{Q}\mathcal{G}_{\lambda_1, \mu_1}(V)$ there is a $(\lambda'_1, \mu'_1)$-quasigeodesic word over $Y$ representing $v$, where $\lambda'_1, \mu'_1$ depend on $\lambda_1, \mu_1$ and $Y$.
\end{corollary}
\begin{lemma}\label{VtoG}
Let $w$ be a ($\lambda'_1, \mu'_1$)-quasigeodesic word over $Y$. Then if the reduced path $p_{\pi(w)}$ is $(a,b)$-quasigeodesic in $\mathcal{K}$ the (unreduced) path $p_{w}$ is $(a_{\mathcal{K}},b_{\mathcal{K}})$-quasigeodesic in $\mathcal{K}$, where $(a_{\mathcal{K}},b_{\mathcal{K}})$ depend on $a,b, \lambda'_1, \mu'_1$ and $Y$. 

Moreover, $\sigma(w)$ is a $(\lambda_G,\lambda_G)$-quasigeodesic over $S$ in the hyperbolic group $G$, where $\lambda_G,\lambda_G$ depend on $\lambda_1, \mu_1$ and $Y$.
\end{lemma}

\begin{proof}
Consider the generating set $Z=Y \cup V_{\leq 3}$ for $V$ and let $\lambda_Z, \mu_Z$ be such that any $(\lambda'_1, \mu'_1)$-quasigeodesic over $Y$ is $(\lambda_Z, \mu_Z)$-quasigeodesic over $Z$. Let $M=\max\{l_{\mathcal{K}}(p_y) \mid y \in Y\}$. That is, $M$ is the maximal length of a generator in $Y$ with respect to the associated reduced path length in $\mathcal{K}$. We will show the statement in the lemma holds for $(a_{\mathcal{K}},b_{\mathcal{K}})=(a,b+M\mu_Z)$.

 We say that a subpath $s_w$ of $p_w$ is a maximal backtrack if $p_w=p s_w p'$, $s_w$ is homotopic to an empty path (via the elimination of backtrackings), and $s_w$ is not contained in a longer subpath of $p_w$ with the same property. This implies there is a point $A$ on $p_w$ such that $s_w$ starts and ends at $A$, and such a maximal backtrack traces a tree in $\mathcal{K}$.  
We can then write $p_w=a_1s_1 a_2 \dots s_{n-1}a_n$, where $a_i$ are (possibly empty) subpaths of $p_w$ and $s_i$ are maximal backtracks; thus $p_{\pi(w)}=a_1a_2 \dots a_n$. If $l_{\mathcal{K}}(s_i) \leq M\mu_Z$ for all $i$, then the result follows immediately. Otherwise there exists an $s_i$ with $l_{\mathcal{K}}(s_i)>M\mu_Z$, and we claim that we can write $s_i$ in terms of a word over $Z$ that is not a quasigeodesic, which contradicts the assumption that $w$ is quasigeodesic.

To prove the claim, suppose $i(s_i)=f(s_i)=A$. We have two cases: in the first case $A\in G$ then $\pi(s_i)=_V 1$ and $s_i$ corresponds to a subword $v$ of  $w$ for which $l_{\mathcal{K}}(p_{v})\geq M\mu_Z$. But $v$ represents a word over $Z$, so $l_{\mathcal{K}}(p_{v}) \leq l_Z(v) M$, and altogether 
$ M\mu_Z \leq l_{\mathcal{K}}(p_{v}) \leq  l_Z(v) M.$
Since $|v|_Z=0$ and $v$ is a ($\lambda_Z, \mu_Z$)-quasigeodesic word over $Z^*$, $l_Z(v) \leq \mu_Z$, which contradicts $l_Z(v) \geq \mu_Z$ from above.

In the second case $A\notin G$, so take a point $B \in G$ at distance $1$ from $A$ in $\mathcal{K}$ (this can always be done), and modify the word $w$ to get $w'$ over $Z$ so that $p_{w'}$ in $\mathcal{K}$ includes the backtrack $[AB,BA]$ off the path $p_w$. Also modify $s_i$ to obtain a new backtrack $s'_i$. Clearly $\pi(p_w)=\pi(p_{w'})$ and $\pi(s_i)=\pi(s'_i)$, and $s'_i$ becomes a maximal backtrack of $p_{w'}$ which can be written as a word over the generators $Z$ that represents the trivial element in $V$. We can the apply the argument from the first case.

The fact that $\sigma(w)$ is a $(\lambda_G,\mu_G)$-quasigeodesic over $S$ in the hyperbolic group $G$ follows from the fact that $\mathcal{K}$ and $\Gamma(G,S)$ are quasi-isometric.
\end{proof}

\bibliographystyle{plainurl}
\bibliography{CiobanuElderIcalp2019}

\appendix
\section{Additional material for Section~\ref{sec:EDT0L}}\label{appendix:EDT0L}

%
%

\begin{proof*}{Proof of Proposition~\ref{prop:projection}}
Since the grammar produces exactly $s-1$ $\#$ symbols, we can deterministically modify the grammar to label each symbol $\#_1,\dots,\#_{s-1}$ in the same space complexity.
Replace each  letter  $c\in C$ by $c^{ij}$ for all $1\leq i\leq j\leq s$ to indicate that the word produced by $c$ will lie in the factor(s)  $u_t, i\leq t\leq j$.
 The new start symbol is  $(S_0)^{1,s}$. For each rule $(c,v)$, if $v$ contains no symbols $\#_t$ then 
replace $(c,v)$  by $(c^{ij},v^{ij})$ for all $1\leq i\leq j\leq s$. If $v=v_0\#_{i_1}v_1 \dots v_{r-1}\#_{i_r}v_r$, 
 for all $i,j$ with $i\leq i_1, j\geq i_r$, replace $(c,v)$ by $(c^{ij},v_0^{i,i_1-1}\#_{i_1}v_1^{i_1,i_2} \dots v_{r-1}^{i_{r-1}i_r}\#_{i_r}v_r^{i_r,j})$
for all $k\leq i_1$, $\ell\geq i_r$. Note these modifications are deterministic so preserve EDT0L. 
%
 \end{proof*}


\section{Additional material for Section~\ref{sec:hyp-intro}}\label{appendix:hyp-intro}

\begin{proof*}{Proof of Proposition~\ref{prop:shortlex-hyp}}
Let $\$$ denote a `padding symbol' that is distinct from $S\cup S'\cup\{\#,\sep\}$, and let 
 \[\mathcal X=\left\{{s\choose t}, {s'\choose t'}, {s\choose \$}, {\$\choose s},{s'\choose \$}, {\$\choose s'},  {\#\choose \#} {\sep\choose \sep}  \mid s,t\in S, s',t'\in S' \right\}.\]
 Define a homomorphism $\psi\colon \mathcal X^*\to (S\cup S'\cup\{\#\})^*$ by 
 \[ \psi\left({x\choose y}\right)=\left\{\begin{array}{llll} 1 & \ \ & x=\$\\x & & x\in S\cup S'\cup\{\#,\sep\}\end{array}\right.\]  
 Then $L_2=\psi^{-1}(L_1)$ is ET0L in \NSPACE$(f(n))$ by Proposition~\ref{prop:closureET0L}. $L_2$ consists of strings of letters which we can view as two parallel strings: the top string is a word from $L_1$ with $\$$ symbols inserted, and the bottom string can be any word in $S\cup S'\cup\{\#,\sep,\$\}$ with $\#,\sep$ occurring  in exactly the same positions as the top string, and by construction there are no two $\$$ symbols in the same position top and bottom.
 
 Let ${\mathcal M}$ be the asynchronous 2-tape automaton which accepts all pairs $(u,v)\in Q_{S,\lambda,\mu}$ 
 with $u=_Gv$ as in Proposition~\ref{prop:reg-hyp}.  For each accept state, add  loops  back to  the start state labeled ${\#\choose \#}$. 
 The new automaton ${\mathcal M}_1$ accepts padded pairs of  $(\lambda, \mu)$-quasigeodesic words  $(u_i,v_i)$ with  $u_i=_Gv_i$,  and each pair is separated by ${\#\choose \#}$. 
 Now make a copy ${\mathcal M}_1'$ of ${\mathcal M}_1$ placing a prime on each letter $s\in S$, and join the accept states of ${\mathcal M}_1$ to the start state of ${\mathcal M}_1'$ by transitions labeled ${\sep\choose \sep}$. The new automaton ${\mathcal M}_2$ accepts some number of padded pairs of words  $(u_i,v_i)\in Q_{S,\lambda,\mu}^2$,  $u_i=_Gv_i$,   followed by some number of padded pairs of words $(u_i',v_i')\in (S')^*$ which are represent $(\lambda,\mu)$-quasigeodesics in $(G,S')$, and $u_i'=_Gv_i'$. 
 
 We now take $L_3=L_2\cap L({\mathcal M}_2)$. This is again ET0L in \NSPACE$(f(n))$ by Proposition~\ref{prop:closureET0L} since $ L({\mathcal M}_2)$ is regular. The language $L_3$ can be seen as pairs of strings, the top string a padded version of a string in $L_1$ and the bottom of the form $w_1\#\dots \#w_r\sep h(z_1)\#\dots \#h(z_r)$ with $w_i=_Gu_i, z_i=_Gv_i$ and $w_i,z_i$ $\lambda,\mu$-quasigeodesics. Moreover, since by hypothesis $u_i=_Gv_i$ we have $w_i=_G z_i$ for all $1\leq i\leq r$.

By Proposition~\ref{prop:closureET0L}, projecting $L_3$ onto the bottom string and deleting $\$$ symbols, which is a homomorphism, 
shows that the language $L_Q$ is ET0L in \NSPACE$(f(n))$.  

Finally, the intersection  $L_4=L_{Q}\cap \mathcal T$ is ET0L in \NSPACE$(f(n))$ since $\mathcal T$ is regular and a subset of $Q_{S,\lambda,\mu}$.
However, since $\mathcal T$ also is in bijection with $G$, we have $w_i$ and $z_i$ are identical words, thus $L_6$ consists of  words of the form $w_1\#\dots \#w_r\sep h(w_1)\#\dots \#h(w_r)$. By Proposition~\ref{prop:copyME}, the language $L_4$ is in fact EDT0L in \NSPACE$(f(n))$, and by Proposition~\ref{prop:projection} projecting onto the prefix gives us $L_{\mathcal T}$ is EDT0L in \NSPACE$(f(n))$.
\end{proof*}


\section{Additional material for Section~\ref{sec:torsionfree}}\label{appendix:torsionfree}

In order to construct canonical representatives one first needs the concept of a \textit{cylinder} of a group element $g$, whose definition is lengthy and not needed here (see Definition 3.1 in \cite{RS95}).
Informally, a cylinder is a narrow neighbourhood of bounded thickness - depending on an integer $T$ called the \textit{criterion} - of any geodesic $[1,g]$; that is, a cylinder contains those points that are `close' to the geodesics connecting $1$ and $g$ in the Cayley graph. 
Denote the cylinder of $g\in G$ depending on $T$ by $C_T(g)$.
 
\begin{proposition}[Lemma 3.2, \cite{RS95}]\label{prop_nbhd}

Let $G$ be a $\delta$-hyperbolic group, $g\in G$ an element, $T$ a criterion, and $h \in C_T(g)$. The following hold:

(1) for every geodesic $[1,g]$, $d(h, [1,g])\leq 2\delta$,

(2) $g^{-1}h \in C_T(g^{-1})$,

(3) any point that lies on a geodesic $[1,g]$ is in $C_T(g)$.

\end{proposition}

For fixed $T$ denote the cylinder of any element $g$ simply by $C(g)$. The next step towards building canonical representatives consists of `cutting' the cylinders into slices.

\begin{definition}\label{def_diff}
Let $w$ be a vertex in $\Gamma(G,S)$ and let $v \in C(w)$. The \textit{left} and \textit{right} neighborhoods of $v$ in $C(w)$ are
$$N_L^{w}(v)= \{ x \in C(w) \mid d(1, x)\leq d(1,v) \ \textrm{and} \ d(v,x)\geq 10\delta\}$$
$$N_R^{w}(v)= \{ x \in C(w) \mid d(1, x)\geq d(1,v) \ \textrm{and} \ d(v,x)\geq 10\delta\}.$$
The \textit{difference} between two elements $u$ and $v$ of $C(w)$ is then
{\small
$$diff_w(u,v)=|N_L^{w}(u)\setminus N_L^{w}(v)|-|N_L^{w}(v) \setminus N_L^{w}(u)|+|N_R^{w}(v) \setminus N_R^{w}(u)|-|N_R^{w}(u) \setminus N_R^{w}(v)|.$$}
\end{definition}

One can show that for any $u,v,t \in C(w)$
$diff_w(u,v) + diff_w(v,t)=diff_w(u,t),$ and $diff_w(u,v)=-diff_w(v,u),$
and define an equivalence relation on $C(w)$ with $u \sim v$ if and only if $diff_w(u,v)=0$. We call the equivalence classes with respect to $\sim$ \textit{slices}, and denote the slice of $v\in C(w) $
by $slice_v^w$. We say that $slice_{v_2}^w$ is \textit{consecutive} to $slice_{v_1}^w$ if $diff_w(v_2,v_1)>0$. Consecutivity between slices is well defined, and we can thus partition each cylinder into slices, which are `small', in the sense that each slice is included in a ball of radius $10\delta$.

\begin{proposition}[Proposition 3.6, \cite{RS95}]\label{prop_slice}
For any $v\in C(w)$, the diameter of $slice_v^w$ is $\leq 10\delta.$
Moreover, If $slice_{v_2}^w$ is consecutive to $slice_{v_1}^w$ and $|v_1|\geq 10 \delta$, then the diameter of $slice_{v_1}^w \cup slice_{v_2}^w$ is $\leq 20\delta +1$.
\end{proposition}
Now let $B_{10\delta}(1)$ be the ball of radius $10 \delta$ around the identity. We define the following equivalence relation on the set of subsets of $B_{10\delta}(1)$:
$$A \sim B \textrm{\ if \ and \ only \ if \ } \exists g\in G \textrm{\ such that \ } gA=B.$$
Every equivalence class is called an \textit{atom}. Choose a set $\{A_1, \dots, A_k \}$ of representatives of all atoms in $B_{10\delta}(1)$. Since every slice $slice_v^w$ in a cylinder has bounded diameter, there exists only one $1 \leq i \leq k$ such that $slice_v^w=gA_i$ for some $g\in G$, and $g$ is unique with this property. The set $A_i$ is called the \textit{model} for $slice_v^w$.

\begin{definition}

For every atom representative $A_i$ choose an (arbitrary) point $a_i$ in $A_i$, called the \textit{center} of $A_i$ and denoted by $ce(A_i)$. Then for every slice $slice_v^w$ the center $ce(slice_v^w)$ is defined by 
$$ce(slice_v^w):= g^{-1}ce(A_i),$$
where $A_i$ is the model for $slice_v^w$ and $g$ the corresponding element sending $A_i$ to $slice_v^w$.

\end{definition}

\begin{definition}[Definition 3.9, \cite{RS95}]
For $x\in G$, $|x|_S \leq 20\delta +1$, the \textit{step} of $x$, written $st(x)$, is a chosen geodesic word on the generators $S$ such that
 $\overline{st(x)}=x$, and $st(x^{-1})=st(x)^{-1}$.
\end{definition}

By having cut the cylinders into slices, defined a `dictionary' of atoms and their centres which canonically assign a center to each slice, one can connect the centres using the fixed steps and get canonical representatives. 

\begin{figure}[h!]
\begin{center}
\includegraphics[scale=0.3]{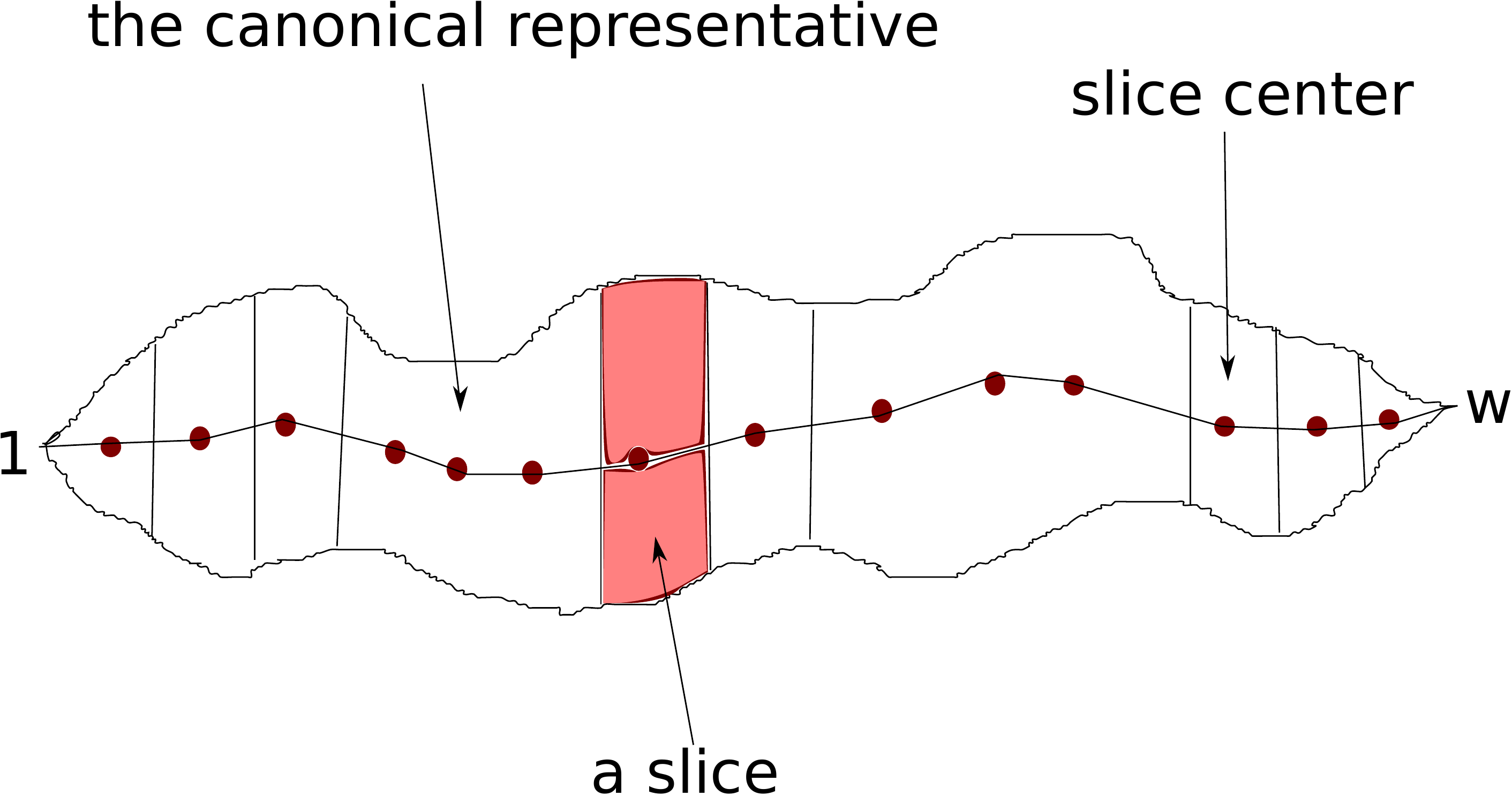}
\caption{The canonical representative of $w$}
\end{center}
\end{figure}

\begin{definition}[Definition 3.10, \cite{RS95}]\label{canrep} Let $w\in G$ be a vertex in the Cayley graph of $G$ and let $\{slice_{v_1}^w, \dots, slice_{v_m=w}^w\}$ be the sequence of consecutive slices of the cylinder $C_T(w)$ depending on criterion $T$. If $|w| \leq 10\delta$ we let $\theta_T(w):=st(w)$. Else we define the \textit{canonical representative} of $w$ to be 
$$\theta_T(w):=st(ce(slice_{v_1}^w)))\star (\star_{i=2}^m st(ce(slice_{v_{i-1}}^w)^{-1}ce(slice_{v_i}^w)))\star st(ce(slice_w^w)^{-1}).$$
\end{definition}

For a fixed $T$ we write $\theta(w)$ instead of $\theta_T(w)$. Canonical representatives satisfy $\theta(w)=_G w$ and $[\theta(w)]^{-1}=\theta(w^{-1})$, and most importantly they are combinatorially \textit{stable} (see Theorem 3.11 in \cite{RS95}); that is, if three elements $w_1, w_2, w_3$ satisfy the triangular relation $w_1w_2w_3=1$, and the cylinders corresponding to the three geodesic sides of the thin triangle agree in balls around the vertices of the triangle. 
then the canonical representatives coincide inside slightly smaller balls around the vertices of the triangle. 

Also essential is the fact that canonical representatives are ($\lambda, \mu$)-quasigeodesics, with $\lambda$ and $\mu$ depending only on $\delta$.

\begin{proposition}[see Proposition 3.4\cite{DahmaniIsrael}]\label{canrep_qg} There are constants $\lambda \geq 1$ and $\mu \geq 0$ depending only on $\delta$ such that for any criterion $T$ and any element $g \in G$ the canonical representative $\theta_T(g)$ of $g$ is a ($\lambda, \mu$)-quasigeodesic.
\end{proposition}

Proposition 3.4\cite{DahmaniIsrael} uses different terminology and applies to the more general setting of relatively hyperbolic groups, which makes it more involved than our case, so we sketch here briefly the proof of Proposition \ref{canrep_qg}, without insisting on the exact constants $\lambda$ and $\mu$, just showing their dependence on $\delta$ only. Recall that a canonical representative $p=\theta_T(g)$ for $g$ consists, as a path, of the concatenation of short ($\leq 20\delta +1$) geodesics between the centres of slices of a cylinder, and by Proposition \ref{prop_slice} each slice is contained in a ball of radius $10\delta$. So without loss of generality we may take $q_{-}$ and $q_{+}$ in 
to be centres separated by $n$ slices in the decomposition of a cylinder. Then $\ell ( q ) \leq (20\delta +1)n$ by the triangle inequality. By Definition \ref{def_diff} we have $diff(q_{-},q_{+})\geq n$, so at least one of the terms in $diff(q_{-},q_{+})$ is greater than $n/4$. Suppose $|N_L^{g}(q_{-})\setminus N_L^{g}(g_{+})| \geq n/4.$ Then by Proposition \ref{prop_nbhd} (1) a cylinder is contained in the union of balls of radius $2\delta$ centred on a geodesic between $1$ and $g$ that intersects $N_L^{g}(q_{-})\setminus N_L^{g}(g_{+})$, so if we let $M=|B(2\delta)|$, we get that $d(q_{-},q_{+})M \geq |N_L^{g}(q_{-})\setminus N_L^{g}(g_{+})| \geq n/4$. This shows that $n \leq 4 d(q_{-},q_{+})M$, so $\ell ( q ) \leq (20\delta +1)4 d(q_{-},q_{+})M$ and therefore $p$ is a quasigeodesic.

We also state the result below, due to Rips and Sela in \cite{RS95}, which is instrumental in getting the solutions in torsion-free hyperbolic groups from free groups. Let us first mention that $\lambda(\Phi):=\lambda(\Phi)(\delta, q)$ is a constant depending on $\delta$ and linearly on $q$, and its explicit formula can be found on page 502 in \cite{RS95}.

 \begin{theorem}[Theorem 4.2, Corollary 4.4 \cite{RS95}]\label{CylStab}
 Let $G$ be a torsion-free $\delta$-hyperbolic group generated by set $S$, and let $\Phi$ be a system of equations of the form 
 $$z_{i(j,1)}z_{i(j,2)}z_{i(j,3)}=1,$$
 where $1\leq j \leq q$. There exists an effectively computable set of criteria $\{T_1, \dots, T_{\lambda(\Phi)} \}$ and a constant $p:=p(\delta, q)$ depending on $\delta$ and linearly on $q$ such that:
 
 If $(g_1, \dots, g_l) \in \text{Sol}_G(\Phi)$, then there exists an index $m_0 \in \{1, \dots, \lambda(\Phi) \}$ and there exist $y_a^j, c_a^j \in F(S)$, $1 \leq j \leq q, 1\leq a\leq 3$, such that $|c_a^j| \leq p$ and $c_1^jc_2^jc_3^j=_G 1$, for which
  $$\theta_{T_{m_0}}(g_{i(j,1)})=y_1^jc_1^j (y_2^j)^{-1},$$ 
   $$\theta_{T_{m_0}}(g_{i(j,2)})=y_2^jc_2^j (y_3^j)^{-1},$$
    $$\theta_{T_{m_0}}(g_{i(j,3)})=y_3^jc_3^j (y_1^j)^{-1}.$$
   \end{theorem}
 
Now suppose $(g_1, \dots, g_l) \in \text{Sol}_G(\Phi)$. Then by Theorem \ref{CylStab} there is a tuple $(w_1, \dots, w_l)$, $w_a \in S^*$ a canonical representative of $g_a$ which has the form $y_a^jc_a^j (y_{a+1}^j)^{-1}$ with $y_a^j, c_a^j \in F(S)$, and all the $y_a^j$ are solutions to one of the systems $\Psi_i$ in Step 3 of the algorithm. Thus to any solution of $\text{Sol}_G(\Phi)$ there is at least one tuple of words in $\mathcal{S}$ which projects to it.

\section{Additional material for Section~\ref{sec:torsion}}\label{appendix:torsion}

For any metric space $X$ and constant $d$, the {\em Rips complex} with parameter $d$ is the simplicial complex whose
vertices are the points of $X$ and whose simplices are the finite subsets of $X$ whose diameter is at
most $d$. More specifically, we will need the Rips complex which is based on the Cayley graph of a hyperbolic group, as follows.

\begin{definition}Let $G$ be a hyperbolic group with finite generating set $X$. For a fixed constant $d$, the Rips complex $\mathcal{P}_d(G)$ is a simplicial complex defined as the collection of sets with diameter (with respect to the $X$-distance in the Cayley graph of $G$) less than $d$:
$$\mathcal{P}_d(G):=\{Y\subset G \mid Y\neq \emptyset, \textrm{ \ diam}_X(Y) \leq d\},$$
where each set $Y \in \mathcal{P}_d(G)$ of cardinality $k+1$ is identified with a $k$-simplex whose vertex set is $Y$.
\end{definition}

The Rips complex of a hyperbolic group has several important properties, some of which are relevant to this paper. The group $G$ acts properly discontinuously on the Rips complex $\mathcal{P}_d(G)$, the quotient $\mathcal{P}_d(G)/G$ is compact, and $\mathcal{P}_d(G)$ is contractible. The fact that $\mathcal{P}_d(G)/G$ is compact is essential to showing the group of paths $V$ in Section \ref{sec:torsion} is a finite graph of finite groups, and therefore virtually free.

\begin{definition}[Barycentric subdivision]

\begin{enumerate}\item[(i)] For a simplex $\tau=\{v_0, v_1, \dots, v_q\}$ of dimension $q$ in Euclidean space, we define its \emph{barycentre} to be the point $b_{\tau}:= \frac{1}{q+1}(v_0+ \dots v_q)$. 

\item[(ii)] For two simplices $\alpha, \beta$ in Euclidean space we write $\alpha<\beta$ to denote that $\alpha$ is a face of $\beta$. 

\item[(iii)] The \emph{barycentric subdivision} $B_{\sigma}$ of a simplicial complex $\sigma$ is the collection of all simplices whose vertices are $b_{\sigma_0}, \dots, b_{\sigma_r}$ for some sequence $\sigma_0 < \dots <\sigma_r$ in $\sigma$. Thus the set of vertices in $B_{\sigma}$ is the set of all barycentres of simplices of $\sigma$, $B_{\sigma}$ has the same dimension as $\sigma$, and any vertex in $B_{\sigma} \setminus \sigma$ is connected to a vertex in $\sigma$.
\end{enumerate}
\end{definition}

Let $P_{50\delta}(G)$ be Rips complex whose set of vertices is $G$, and whose simplices are subsets of $G$ of diameter at most $50\delta$. Then let $\mathcal{B}$ be the barycentric subdivision of $P_{50\delta}(G)$ and let $\mathcal{K}:=\mathcal{B}^1$ the $1$-skeleton of $\mathcal{B}$. Thus by construction the vertices of $\mathcal{K}$ (and $\mathcal{B}$) are in $1$-to-$1$ correspondence with the simplices of $P_{50\delta}(G)$, so we can identify $G$, viewed as a set of $0$-simplices in $\mathcal{B}$, with a subset of vertices of $\mathcal{K}$.  

\begin{remark}\label{rmkKG} By construction, the graphs $\mathcal{K}$ and $\Gamma(G,S)$ are quasi-isometric, and any vertex in $\mathcal{K}$ that is not in $G$ is at distance one (in $\mathcal{K}$) from a vertex in $G$.
\end{remark}

\begin{proposition}[Proposition 9.8, \cite{DG}]\label{prop:K}
Let $G$ be a hyperbolic group, $\mathcal{K}$ as above, and $\Phi$ a system of triangular equations. Let $\lambda_0$ a constant depending on $\delta$, $\mu_0=8$ and $b$ a constant depending on $\delta$ and the number of equations in $\Phi$.

 If $(g_1, \dots, g_l) \in \text{Sol}_G(\Phi)$, then for any variable $z$ and corresponding solution $g_z$ there exists a $(\lambda_0,\mu_0)$-quasigeodesic path $\gamma_z$ joining $1$ to $g_z$ in $\mathcal{K}$ such that $\gamma_{z^{-1}}=^{g_z^{-1}}\overline{\gamma_z}$ and for each equation $z_1z_2z_3=1$ there is a decomposition of the paths $\gamma_{z_i}$ into subpaths:
 $$\gamma_{z_i}=l_{z_i} c_{z_i} r_{z_i}, \ \ \ l_{z_{i+1}}=^{g_{z_i}^{-1}}\overline{r_{z_i}}$$
 where $i+1$ is computed modulo $3$, and $l(c_{z_i})\leq b$.
 \end{proposition}
 
 For $\lambda_0, \mu_0$ as in Proposition \ref{prop:K}, let $\lambda_1, \mu_1$ be such that any path $\alpha \gamma \alpha'$ in $\mathcal{K}$ is a $(\lambda_1, \mu_1)$-quasigeodesic, where $l(\alpha)=l(\alpha')=1$ and $\gamma$ is a $(\lambda_0, \mu_0)$-guasi-geodesic. Then define 
\begin{equation}\label{qgK}
\mathcal{Q}\mathcal{G}_{\lambda_1, \mu_1}(V):=\{[\gamma] \in V \mid \gamma \textrm{\ is a reduced \ } (\lambda_1, \mu_1)-\textrm{quasigeodesic in \ } \mathcal{K}\},
\end{equation}
and for any $L>0$ let $V_{\leq L}:=\{[\gamma] \in V\mid \gamma \textrm{\ reduced\ and\ } \ell_{\mathcal{K}}(\gamma)\leq L \}.$

 \begin{proposition}[Proposition 9.10, \cite{DG}]\label{prop:V}
Let $G$ be a hyperbolic group, $\mathcal{K}$, $\lambda_1, \mu_1$ as above, and $\Phi$ a system of triangular equations. Let $\kappa$ be a constant that depends on $\delta$ and the number of equations. 
\begin{enumerate}
\item[(i)] If $(g_1, \dots, g_l) \in \text{Sol}_G(\Phi)$, then for any variable $z$ and corresponding solution $g_z \in G$ there exists $v_{z} \in \mathcal{Q}\mathcal{G}_{\lambda_1, \mu_1}(V)$ such that $f(v_{z})=g_z$ and for each equation $\mathcal{E}$, written $z_1z_2z_3=1$, there exist $l_i:=l_{z_i, \mathcal{E}} \in \mathcal{Q}\mathcal{G}_{\lambda_1, \mu_1}(V)$ and $c_i:=c_{z_i,\mathcal{E}} \in V_{\leq \kappa}$ such that
$$ v_{g_{z_i}}= l_i c_i l^{-1}_{i+1} \in V,$$
where $i+1$ is computed modulo $3$, and $f(c_1 c_2 c_3)=1$ in $G$.
\item[(ii)] Conversely, given any family of elements $v_z$ of $V$, for each variable $z$ in $\mathcal{E}$, and $l_i:=l_{z_i, \mathcal{E}} \in V$ and $c_i:=c_{z_i,\mathcal{E}} \in V_{\leq \kappa}$ such that $ v_{z_i}= l_i c_i l^{-1}_{i+1} \in V$ and $f(c_1 c_2 c_3)=_G1$, then $g_z=f(v_z)$ is a solution of the system $\mathcal{E}$ in $G$.

\end{enumerate}
 \end{proposition}

\end{document}